\documentclass[10pt]{amsart}

\usepackage{amssymb}

\usepackage{a4wide}

\newtheorem{theorem}{Theorem}[section]

\newtheorem{corollary}[theorem]{Corollary}
\newtheorem{proposition}[theorem]{Proposition}

\theoremstyle{definition}
\newtheorem{definition}[theorem]{Definition}
\newtheorem{example}[theorem]{Example}

\theoremstyle{remark}
\newtheorem{remark}[theorem]{Remark}

\numberwithin{equation}{section}

\begin{document}

\title{Symmetric Differentiation on Time Scales}

\thanks{This is a preprint of a paper whose final and definite form 
will be published in \emph{Applied Mathematics Letters}.
Submitted 30-Jul-2012; revised 07-Sept-2012; accepted 10-Sept-2012.}


\author[A. M. C. Brito da Cruz]{Artur M. C. Brito da Cruz}

\address{Escola Superior de Tecnologia de Set\'{u}bal,
Estefanilha, 2910-761 Set\'{u}bal, Portugal
\and
Center for Research and Development in Mathematics and Applications,
University of Aveiro, 3810-193 Aveiro, Portugal}

\email{artur.cruz@estsetubal.ips.pt}


\author[N. Martins]{Nat\'{a}lia Martins}

\address{Center for Research and Development in Mathematics and Applications,
Department of Mathematics, University of Aveiro, 3810-193 Aveiro, Portugal}

\email{natalia@ua.pt}


\author[D. F. M. Torres]{Delfim F. M. Torres}

\address{Center for Research and Development in Mathematics and Applications,
Department of Mathematics, University of Aveiro, 3810-193 Aveiro, Portugal}

\email{delfim@ua.pt}



\date{}


\begin{abstract}
We define a symmetric derivative on an arbitrary nonempty closed subset
of the real numbers and derive some of its properties. It is shown that
real-valued functions defined on time scales that are neither delta
nor nabla differentiable can be symmetric differentiable.
\end{abstract}


\maketitle


\section{Introduction}

Symmetric properties of functions are very useful in a large number of
problems. Particularly in the theory of trigonometric series, applications
of such properties are well known \cite{Zygmund}. Differentiability
is one of the most important properties in the theory of functions
of real variables. However, even simple functions such as
\begin{equation}
\label{ts:1}
f\left( t\right)
=\left \vert t\right \vert,
\quad g\left( t\right)
= \begin{cases}
t\sin \frac{1}{t} \, , & t\neq 0 \\
0 \, , &  t=0,
\end{cases}
\quad h\left( t\right) = \frac{1}{t^{2}} \, , t \neq 0,
\end{equation}
do not have (classical) derivative at $t=0$.
Authors like Riemann, Schwarz, Peano, Dini,
and de la Vall\'{e}e-Poussin, extended the
classical derivative in different ways,
depending on the purpose \cite{Zygmund}.
One of those notions is the symmetric derivative:
\begin{equation}
\label{eq:sd:r}
f^{s}\left( t\right)
=\lim_{h\rightarrow 0}\frac{f\left( t+h\right) -f\left(t-h\right) }{2h}.
\end{equation}
While the functions in \eqref{ts:1} do not have ordinary derivatives at $t=0$,
they have symmetric derivatives:
$f^{s}\left( 0\right) = g^{s}\left( 0\right) = h^{s}\left( 0\right) =0$.
For a deeper understanding of the symmetric derivative and its properties,
we refer the reader to the specialized monograph \cite{Thomson}. Here we note that
the symmetric quotient $\left(f\left( t+h\right) -f\left( t-h\right)\right)/(2h)$
has, in general, better convergence properties than the ordinary difference
quotient \cite{Serafin}, leading naturally to the so-called
$h$-symmetric quantum calculus \cite{Kac}. A more recent theory
is the general time-scale calculus.
In 1988, Hilger introduced the calculus on time scales
as a generalization of continuous and discrete time theories,
obviating the need for separate proofs
and highlighting the differences between them \cite{Hilger,Bohner:1}.
Here we introduce the notion of symmetric derivative on time scales,
initiating the corresponding theory and putting into context some
of the recent results found in the literature.

The article is organized as follows. In Section~\ref{ts:sec:pre}
we review the necessary concepts and we fix notations.
The results are then given in Section~\ref{ts:sec:dif},
where we define the time scale symmetric derivative
and derive some of its properties.
Applications are found in
the context of quantum calculus \cite{Kac}.
Finally, we show in Section~\ref{ts:sec:fc}
that the new symmetric derivative is a generalization
of the diamond-$\alpha$ derivative \cite{Sheng},
which brings new insights.


\section{Preliminary notions and notations}
\label{ts:sec:pre}

A nonempty closed subset of $\mathbb{R}$ is called
a time scale and is denoted by $\mathbb{T}$. We assume that a time scale
has the topology inherited from $\mathbb{R}$ with the standard topology.
Two jump operators are considered: the forward jump operator
$\sigma : \mathbb{T} \rightarrow \mathbb{T}$, defined by
$\sigma \left( t\right) :=\inf \left \{ s\in \mathbb{T}:s>t\right \}$
with $\inf \emptyset =\sup \mathbb{T}$ (i.e., $\sigma \left( M\right) =M$ if
$\mathbb{T}$ has a maximum $M$), and the backward jump operator
$\rho :\mathbb{T} \rightarrow \mathbb{T}$ defined by
$\rho \left( t\right) :=\sup \left \{ s\in \mathbb{T}:s<t\right \}$
with $\sup \emptyset =\inf \mathbb{T}$ (i.e.,
$\rho \left( m\right) =m$ if $\mathbb{T}$ has a minimum $m$).
A point $t\in \mathbb{T}$ is said to be right-dense, right-scattered,
left-dense or left-scattered if $\sigma \left( t\right) =t$,
$\sigma \left( t\right) >t$, $\rho \left( t\right) =t$
or $\rho \left( t\right) <t$, respectively. A point $t\in \mathbb{T}$
is dense if it is right and left dense;
isolated if it is right and left scattered.
If $\sup\mathbb{T}$ is finite and left-scattered, then we define
$\mathbb{T}^\kappa := \mathbb{T}\setminus \{\sup\mathbb{T}\}$,
otherwise $\mathbb{T}^\kappa :=\mathbb{T}$;
if $\inf\mathbb{T}$ is finite and right-scattered, then
$\mathbb{T}_\kappa := \mathbb{T}\setminus\{\inf\mathbb{T}\}$,
otherwise $\mathbb{T}_\kappa := \mathbb{T}$. We set
$\mathbb{T}_{\kappa}^{\kappa}:=\mathbb{T}_{\kappa}\cap \mathbb{T}^{\kappa}$
and we denote $f\circ\sigma$ by $f^{\sigma}$ and $f\circ\rho$ by $f^{\rho}$.


\section{Main Results}
\label{ts:sec:dif}

In quantum calculus, the $h$-symmetric difference
and the $q$-symmetric difference, $h>0$ and $0<q<1$, are defined by
\begin{equation}
\label{eq:sd:hc}
\tilde{D}_{h}\left( t\right) =\frac{f\left( t+h\right) -f\left( t-h\right) }{2h}
\end{equation}
and
\begin{equation}
\label{eq:sd:qc}
\tilde{D}_{q}\left( t\right) =\frac{f\left( qt\right) -f\left(
q^{-1}t\right) }{\left( q-q^{-1}\right) t}, \quad t\neq 0,
\end{equation}
respectively \cite{Kac}. Here we propose a general notion
of symmetric derivative on time scales
that encompasses all the three definitions
\eqref{eq:sd:r}, \eqref{eq:sd:hc}, and \eqref{eq:sd:qc}.

\begin{definition}
We say that a function $f:\mathbb{T} \rightarrow \mathbb{R}$
is symmetric continuous at $t\in $ $\mathbb{T}$ if, for any
$\varepsilon >0$, there exists a neighborhood $U_{t} \subset \mathbb{T}$ of $t$
such that for all $s\in U_{t}$ for which  $2t-s \in U_{t}$
one has $\left \vert f\left( s\right) -f\left( 2t-s\right) \right \vert \leqslant \varepsilon$.
\end{definition}

It is easy to see that continuity implies symmetric continuity.

\begin{proposition}
Let $\mathbb{T}$ be a time scale. If $f: \mathbb{T} \rightarrow \mathbb{R}$
is a continuous function, then $f$ is symmetric continuous.
\end{proposition}

\begin{proof}
Since $f$ is continuous at $t\in \mathbb{T}$, then, for any $\varepsilon >0$,
there exists a neighborhood $U_{t}$ such that
$\left\vert f\left( s\right) -f\left( t\right) \right\vert <\frac{\varepsilon}{2}$
and $\left\vert f\left( 2t-s\right) -f\left( t\right) \right\vert
<\frac{\varepsilon }{2}$ for all $s\in U_{t}$ for which $2t-s \in  U_{t}$. Thus,
$\left\vert f\left( s\right) -f\left( 2t-s\right) \right\vert
\leqslant \left\vert f\left( s\right) -f\left( t\right) \right\vert +\left\vert
f\left( t\right) -f\left( 2t-s\right) \right\vert <\varepsilon$.
\end{proof}

The next example shows that symmetric continuity does not imply continuity.

\begin{example}
Consider the function $f:\mathbb{R} \rightarrow \mathbb{R}$ defined by
\[
f\left( t\right) =
\begin{cases}
0 \, , &  \text{ if } \, t\neq 0, \\
1 \, , &  \text{ if } \, t=0.
\end{cases}
\]
Function $f$ is symmetric continuous at $0$: for any $\varepsilon>0$,
there exists a neighborhood $U_{t}$ of $t=0$ such that
$\left\vert f\left( s\right) -f\left( -s\right) \right\vert =0<\varepsilon$
for all $s\in U_{t}$ for which $-s\in  U_{t}$.
However, $f$ is not continuous at $0$.
\end{example}

\begin{definition}
Let $f:\mathbb{T} \rightarrow \mathbb{R}$
and $t\in \mathbb{T}_{\kappa}^{\kappa}$.
The symmetric derivative of $f$ at $t$,
denoted by $f^{\diamondsuit }\left( t\right)$, is the real number
(provided it exists) with the property that, for any $\varepsilon >0$,
there exists a neighborhood $U \subset \mathbb{T}$ of $t$ such that
\begin{equation*}
\left \vert \left[ f^{\sigma }\left( t\right) -f\left( s\right) +f\left(
2t-s\right) -f^{\rho }\left( t\right) \right] -f^{\diamondsuit }\left(
t\right) \left[ \sigma \left( t\right) +2t-2s-\rho \left( t\right) \right]
\right \vert
\leqslant \varepsilon \left \vert \sigma \left( t\right)
+2t-2s-\rho \left(t\right) \right \vert
\end{equation*}
for all $s\in U$ for which  $2t-s \in U$.
A function $f$ is said to be symmetric differentiable provided
$f^{\diamondsuit }\left(t\right)$ exists for all
$t\in \mathbb{T}_{\kappa}^{\kappa}$.
\end{definition}

Some useful properties of the symmetric derivative
are given in Theorem~\ref{ts:propriedade}.

\begin{theorem}
\label{ts:propriedade}
Let $f:\mathbb{T} \rightarrow \mathbb{R}$ and $t\in \mathbb{T}_{\kappa}^{\kappa}$.
The following holds:

\begin{enumerate}
\item[(i)] Function $f$ has at most one symmetric derivative at $t$.

\item[(ii)] If $f$ is symmetric differentiable at $t$, then $f$ is symmetric
continuous at $t$.

\item[(iii)] If $f$ is continuous at $t$ and $t$ is not dense,
then $f$ is symmetric differentiable at $t$ with
$f^{\diamondsuit }\left( t\right)
=\frac{f^\sigma(t) -f^\rho(t)}{\sigma \left( t\right)
-\rho\left( t\right)}$.

\item[(iv)] If $t$ is dense, then $f$ is symmetric differentiable at $t$ if and
only if the limit
$\lim_{s\rightarrow t}\frac{f\left( 2t-s\right)-f\left( s\right)}{2t-2s}$
exists as a finite number. In this case
$f^{\diamondsuit }\left( t\right)
=\lim_{s\rightarrow t}\frac{f\left( 2t-s\right)-f\left(s\right)}{2t-2s}
=\lim_{h\rightarrow 0}\frac{f\left( t+h\right) -f\left( t-h\right) }{2h}$.

\item[(v)] If $f$ is symmetric differentiable and continuous at $t$, then
$f^{\sigma }\left( t\right) =f^{\rho }\left( t\right) +f^{\diamondsuit}\left( t\right)
\left[ \sigma \left( t\right) -\rho \left( t\right) \right]$.
\end{enumerate}
\end{theorem}

\begin{proof}
(i) Suppose that $f$ has two symmetric derivatives at $t$,
$f_{1}^{\diamondsuit }\left( t\right)$ and
$f_{2}^{\diamondsuit }\left(t\right)$.
Then, there exists a neighborhood $U_{1}$ of $t$ such that
\begin{equation*}
\left \vert \left[ f^{\sigma }\left( t\right) -f\left( s\right)
+f\left(2t-s\right) -f^{\rho }\left( t\right) \right]
-f_{1}^{\diamondsuit }\left(t\right) \left[ \sigma \left( t\right)
+2t-2s-\rho \left( t\right) \right]\right \vert
\leqslant \frac{\varepsilon }{2}\left \vert \sigma \left( t\right)
+2t-2s-\rho \left( t\right) \right \vert
\end{equation*}
for all $s\in U_{1}$ for which $2t-s\in U_{1}$,
and a neighborhood $U_{2}$ of $t$ such that
\begin{equation*}
\left \vert \left[ f^{\sigma }\left( t\right) -f\left( s\right) +f\left(
2t-s\right) -f^{\rho }\left( t\right) \right] -f_{2}^{\diamondsuit }\left(
t\right) \left[ \sigma \left( t\right) +2t-2s-\rho \left( t\right) \right]
\right \vert
\leqslant\frac{\varepsilon }{2}\left \vert \sigma \left( t\right)
+2t-2s-\rho \left( t\right) \right \vert
\end{equation*}
for all $s\in U_{2}$ for which $2t-s\in U_{2}$. Therefore,
for all $s\in U_{1}\cap U_{2}$ for which $2t-s \in U_{1}\cap U_{2}$,
\begin{equation*}
\begin{split}
\bigg{|}&f_{1}^{\diamondsuit}\left( t\right) -f_{2}^{\diamondsuit }\left(
t\right) \bigg{|} =\left \vert \left[ f_{1}^{\diamondsuit }\left( t\right)
-f_{2}^{\diamondsuit }\left( t\right) \right] \frac{\sigma \left( t\right)
+2t-2s-\rho \left( t\right) }{\sigma \left( t\right) +2t-2s-\rho \left(
t\right) }\right \vert  \\
&=\frac{1}{\left \vert \sigma \left( t\right) +2t-2s-\rho \left(
t\right) \right \vert } \bigg{|}\left[ f^{\sigma }\left( t\right) -f\left( s\right) +f\left(
2t-s\right) -f^{\rho }\left( t\right) \right] -f_{2}^{\diamondsuit }\left(
t\right) \left[ \sigma \left( t\right) +2t-2s-\rho \left( t\right) \right]\\
&\qquad -\left[ f^{\sigma }\left( t\right) -f\left( s\right) +f\left( 2t-s\right)
-f^{\rho }\left( t\right) \right] +f_{1}^{\diamondsuit }\left( t\right)
\left[ \sigma \left( t\right) +2t-2s-\rho \left( t\right) \right] \bigg{|}\\
&\leqslant \frac{1}{\left \vert \sigma \left( t\right) +2t-2s-\rho \left(
t\right) \right \vert } \bigg{(}\left \vert \left[ f^{\sigma }\left( t\right) -f\left(
s\right) +f\left( 2t-s\right) -f^{\rho }\left( t\right) \right]
-f_{2}^{\diamondsuit }\left( t\right) \left[ \sigma \left( t\right)
+2t-2s-\rho \left( t\right) \right] \right \vert  \\
&\qquad +\left \vert \left[ f^{\sigma }\left( t\right) -f\left( s\right) +f\left(
2t-s\right) -f^{\rho }\left( t\right) \right] -f_{1}^{\diamondsuit }\left(
t\right) \left[ \sigma \left( t\right) +2t-2s-\rho \left( t\right) \right]
\right \vert \bigg{)} \\
&\leqslant \varepsilon.
\end{split}
\end{equation*}

\noindent (ii) From hypothesis, for any $\epsilon^{*}>0$,
there exists a neighborhood $U$ of $t$ such that
\begin{equation*}
\left \vert \left[ f^{\sigma }\left( t\right) -f\left( s\right) +f\left(
2t-s\right) -f^{\rho }\left( t\right) \right] -f^{\diamondsuit }\left(
t\right) \left[ \sigma \left( t\right) +2t-2s-\rho \left( t\right) \right]
\right \vert
\leqslant \varepsilon ^{\ast }\left \vert \sigma \left( t\right)
+2t-2s-\rho \left( t\right) \right \vert
\end{equation*}
for all $s\in U$ for which  $2t-s \in U$.
Therefore, for all $s\in U \cap \left] t-\varepsilon^{\ast },t
+\varepsilon ^{\ast }\right[$,
\begin{equation*}
\begin{split}
\vert  f\left( 2t-s\right)  -f\left( s\right) \vert
&\leqslant \left \vert \left[ f^{\sigma }\left( t\right) -f\left( s\right)
+f\left( 2t-s\right) -f^{\rho }\left( t\right) \right] -f^{\diamondsuit
}\left( t\right) \left[ \sigma \left( t\right) +2t-2s-\rho \left( t\right)
\right] \right \vert  \\
&\qquad +\left \vert \left[ f^{\sigma }\left( t\right) -f^{\rho }\left( t\right)
\right] -f^{\diamondsuit }\left( t\right) \left[ \sigma \left( t\right)
+2t-2s-\rho \left( t\right) \right] \right \vert  \\
&\leqslant \left \vert \left[ f^{\sigma }\left( t\right) -f\left( s\right)
+f\left( 2t-s\right) -f^{\rho }\left( t\right) \right] -f^{\diamondsuit
}\left( t\right) \left[ \sigma \left( t\right) +2t-2s-\rho \left( t\right)
\right] \right \vert  \\
&\qquad +\left \vert \left[ f^{\sigma }\left( t\right) -f\left( t\right) +f\left(
t\right) -f^{\rho }\left( t\right) \right] -f^{\diamondsuit }\left( t\right)
\left[ \sigma \left( t\right) +2t-2t-\rho \left( t\right) \right]
\right \vert  \\
&\qquad +2\left \vert f^{\diamondsuit }\left( t\right) \right \vert \left \vert
t-s\right \vert\\
&\leqslant \varepsilon ^{\ast }\left \vert \sigma \left( t\right)
+2t-2s-\rho \left( t\right) \right \vert +\varepsilon ^{\ast }\left \vert
\sigma \left( t\right) +2t-2t-\rho \left( t\right) \right \vert +2\left \vert
f^{\diamondsuit }\left( t\right) \right \vert \left \vert t-s\right \vert  \\
&\leqslant \varepsilon ^{\ast }\left \vert \sigma \left( t\right) -\rho
\left( t\right) \right \vert +2\varepsilon ^{\ast }\left \vert t-s\right \vert
+\varepsilon ^{\ast }\left \vert \sigma \left( t\right) -\rho \left( t\right)
\right \vert +2\left \vert f^{\diamondsuit }\left( t\right) \right \vert
\left \vert t-s\right \vert  \\
&= 2\varepsilon ^{\ast }\left \vert \sigma \left( t\right) -\rho
\left( t\right) \right \vert +2\left( \varepsilon ^{\ast }+\left \vert
f^{\diamondsuit }\left( t\right) \right \vert \right) \left \vert
t-s\right \vert  \\
&\leqslant 2\varepsilon ^{\ast }\left \vert \sigma \left( t\right) -\rho \left(
t\right) \right \vert +2\left( \varepsilon ^{\ast }+\left \vert
f^{\diamondsuit }\left( t\right) \right \vert \right) \varepsilon ^{\ast } \\
&= 2\varepsilon ^{\ast }\left[ \left \vert \sigma \left( t\right)
-\rho \left( t\right) \right \vert +\varepsilon ^{\ast }+\left \vert
f^{\diamondsuit }\left( t\right) \right \vert \right],
\end{split}
\end{equation*}
proving that $f$ is symmetric continuous at $t$.

\noindent (iii) Suppose that $t\in\mathbb{T}_{\kappa}^{\kappa}$
is not dense and $f$ is continuous at $t$. Then,
\begin{equation*}
\lim_{s\rightarrow t}\frac{f^{\sigma }\left( t\right) -f\left( s\right)
+f\left( 2t-s\right) -f^{\rho }\left( t\right) }{\sigma \left( t\right)
+2t-2s-\rho \left( t\right) }=\frac{f^{\sigma }\left( t\right)
-f^{\rho}\left( t\right) }{\sigma \left( t\right) -\rho \left( t\right)}.
\end{equation*}
Hence, for any $\varepsilon >0$, there exists a neighborhood $U$ of $t$ such that
\begin{equation*}
\left \vert \frac{f^{\sigma }\left( t\right) -f\left( s\right) +f\left(
2t-s\right) -f^{\rho }\left( t\right) }{\sigma \left( t\right) +2t-2s-\rho
\left( t\right) }-\frac{f^{\sigma }\left( t\right)
-f^{\rho }\left( t\right)}{\sigma \left( t\right)
-\rho \left( t\right) }\right \vert \leqslant \varepsilon
\end{equation*}
for all $s\in U$ for which  $2t-s \in U$. It follows that
\begin{multline*}
\left \vert \left[ f^{\sigma }\left( t\right) -f\left( s\right)
+f\left(2t-s\right) -f^{\rho }\left( t\right) \right]
-\frac{f^{\sigma }\left(t\right) -f^{\rho }\left( t\right)}{\sigma \left( t\right)
-\rho \left(t\right) }\left[ \sigma \left( t\right)
+2t-2s-\rho \left( t\right) \right]\right\vert\\
\leqslant \varepsilon \left[ \sigma \left( t\right)
+2t-2s-\rho \left(t\right) \right],
\end{multline*}
which proves that
$f^{\diamondsuit }\left( t\right) =\left(f^{\sigma }\left( t\right)
-f^{\rho}\left( t\right)\right)/\left(\sigma \left( t\right) -\rho \left( t\right)\right)$.

\noindent (iv) Assume that $f$ is symmetric differentiable at $t$ and $t$ is dense. Let
$\varepsilon >0$ be given. Then, there exists a neighborhood $U$ of $t$ such that
\begin{equation*}
\left \vert \left[ f^{\sigma }\left( t\right) -f\left( s\right)
+f\left(2t-s\right) -f^{\rho }\left( t\right) \right]
-f^{\diamondsuit }\left(t\right) \left[ \sigma \left( t\right)
+2t-2s-\rho \left( t\right) \right]\right \vert
\leqslant \varepsilon \left \vert \sigma \left( t\right)
+2t-2s-\rho \left(t\right) \right \vert
\end{equation*}
for all $s\in U$ for which  $2t-s \in U$.
Since $t$ is dense,
$\left \vert \left[ -f\left( s\right) +f\left( 2t-s\right) \right]
-f^{\diamondsuit }\left( t\right) \left[ 2t-2s\right] \right \vert \leqslant
\varepsilon \left \vert 2t-2s\right \vert$
for all $s\in U$ for which  $2t-s \in U$. It follows that
$\left \vert \frac{f\left( 2t-s\right)-f\left( s\right)}{2t-2s}
-f^{\diamondsuit }\left( t\right) \right \vert \leqslant \varepsilon$
for all $s\in U$ with $s\neq t$. Therefore, we get the desired result:
$f^{\diamondsuit }\left( t\right)
=\lim_{s\rightarrow t}\frac{f\left( 2t-s\right)-f\left(s\right)}{2t-2s}$.
Conversely, let us suppose that $t$ is dense and the limit
$\lim_{s\rightarrow t}\frac{f\left( 2t-s\right)-f\left( s\right)}{2t-2s} =: L$
exists. Then, there exists a neighborhood $U$ of $t$ such that
$\left \vert \frac{f\left( 2t-s\right)-f\left( s\right)}{2t-2s}
-L\right \vert \leqslant \varepsilon$
for all  $s\in U$ for which  $2t-s \in U$.
Because $t$ is dense, we have
$\left \vert \frac{f^{\sigma }\left( t\right) -f\left( s\right) +f\left(
2t-s\right) -f^{\rho }\left( t\right) }{\sigma \left( t\right) +2t-2s-\rho
\left( t\right) }-L\right \vert \leqslant \varepsilon$.
Therefore,
$$
\left \vert \left[ f^{\sigma }\left( t\right) -f\left( s\right)
+f\left( 2t-s\right) -f^{\rho }\left( t\right) \right] -L\left[ \sigma
\left( t\right) +2t-2s-\rho \left( t\right) \right] \right \vert
\leqslant \varepsilon \left \vert \sigma \left( t\right) +2t-2s-\rho \left(
t\right) \right \vert,
$$
which leads us to the conclusion that $f$ is symmetric differentiable and
$f^{\diamondsuit }\left( t\right)  = L$.
Note that if we use the substitution $s=t+h$, then
$f^{\diamondsuit }\left( t\right)
=\lim_{h\rightarrow 0}\frac{f\left( t+h\right) -f\left( t-h\right) }{2h}$.

\noindent (v) If $t$ is a dense point,
then $\sigma \left( t\right) =\rho \left(t\right) $ and
$f^{\sigma }\left( t\right) =f^{\rho }\left( t\right)
+f^{\diamondsuit }\left( t\right) \left[ \sigma \left( t\right)
-\rho \left(t\right) \right]$.
If $t$ is not dense, and since $f$ is continuous, then
$f^{\diamondsuit }\left( t\right) =\frac{f^{\sigma }\left( t\right)
-f^{\rho }\left( t\right) }{\sigma \left( t\right) -\rho \left( t\right) }
\Leftrightarrow f^{\sigma }\left( t\right) =f^{\rho }\left( t\right)
+f^{\diamondsuit }\left( t\right) \left[ \sigma \left( t\right) -\rho \left(
t\right) \right]$.
\end{proof}

\begin{example}
Let $\mathbb{T} = \mathbb{R}$. Then our symmetric derivative
coincides with the classic symmetric derivative \eqref{eq:sd:r}:
$f^{\diamondsuit} = f^{s}$.
\end{example}

\begin{example}
Let $\mathbb{T} = h \mathbb{Z}$, $h>0$.
Then the symmetric derivative is the symmetric difference operator
\eqref{eq:sd:hc}: $f^{\diamondsuit} = \tilde{D}_{h}$.
\end{example}

\begin{example}
Let  $\mathbb{T=}\overline{q^{\mathbb{Z}}}$, $0<q<1$.
Then the symmetric derivative coincides with the $q$-symmetric
difference operator \eqref{eq:sd:qc}:
$f^{\diamondsuit } = \tilde{D}_{q}$.
\end{example}

\begin{remark}
Independently of the time scale $\mathbb{T}$,
the symmetric derivative of a constant is zero
and the symmetric derivative of the identity function is one.
\end{remark}

\begin{remark}
An alternative way to define the symmetric derivative of $f$
at $t\in\mathbb{T}_{\kappa}^{\kappa}$ consists in saying
that the limit
$f^{\diamondsuit }\left( t\right) =\lim_{s\rightarrow t}\frac{f^{\sigma
}\left( t\right) -f\left( s\right) +f\left( 2t-s\right) -f^{\rho }\left(
t\right) }{\sigma \left( t\right) +2t-2s-\rho \left( t\right) }
=\lim_{h\rightarrow 0}\frac{f^{\sigma }\left( t\right) -f\left( t+h\right)
+f\left( t-h\right) -f^{\rho }\left( t\right) }{\sigma \left( t\right)
-2h-\rho \left( t\right)}$
exists.
\end{remark}

\begin{theorem}
Let $f,g:\mathbb{T} \rightarrow \mathbb{R}$ be two symmetric differentiable functions
at $t\in \mathbb{T}_{\kappa}^{\kappa}$ and $\lambda \in \mathbb{R}$.
The following holds:
\begin{enumerate}
\item[(i)] Function $f+g$ is symmetric differentiable at $t$ with
$\left( f+g\right) ^{\diamondsuit }\left( t\right)
=f^{\diamondsuit }\left(t\right) +g^{\diamondsuit }\left( t\right)$.

\item[(ii)] Function $\lambda f$ is symmetric differentiable at $t$ with
$\left( \lambda f\right)^{\diamondsuit }\left( t\right)
=\lambda f^{\diamondsuit }\left( t\right)$.

\item[(iii)] If $f$ and $g$ are continuous at $t$, then
$fg$ is symmetric differentiable at $t$ with
$\left(fg\right)^{\diamondsuit }\left( t\right)
=f^{\diamondsuit }\left(t\right) g^{\sigma }\left( t\right)
+f^{\rho }\left( t\right) g^{\diamondsuit }\left( t\right)$.

\item[(iv)] If $f$ is continuous at $t$ and
$f^{\sigma }\left( t\right) f^{\rho}\left( t\right) \neq 0$, then
$1/f$ is symmetric differentiable at $t$ with
$\left(\frac{1}{f}\right)^{\diamondsuit }\left( t\right)
=-\frac{f^{\diamondsuit }\left( t\right) }{f^{\sigma }\left(
t\right) f^{\rho}\left( t\right)}$.

\item[(v)] If $f$ and $g$ are continuous at $t$ and
$g^{\sigma }\left( t\right) g^{\rho }\left( t\right) \neq 0$,
then $f/g$ is symmetric differentiable at $t$ with
$\left( \frac{f}{g}\right) ^{\diamondsuit }\left( t\right)
=\frac{f^{\diamondsuit }\left( t\right) g^{\rho }\left( t\right) -f^{\rho }\left(
t\right) g^{\diamondsuit }\left( t\right)}{g^{\sigma }\left( t\right)
g^{\rho }\left( t\right)}$.
\end{enumerate}
\end{theorem}

\begin{proof}
(i) For $t\in \mathbb{T}_{\kappa}^{\kappa}$ we have
\begin{equation*}
\begin{split}
\left( f+g\right)^{\diamondsuit }\left( t\right)
&=\lim_{s\rightarrow t}\frac{\left( f+g\right) ^{\sigma }\left( t\right)
-\left( f+g\right) \left( s\right) +\left( f+g\right) \left( 2t-s\right)
-\left( f+g\right) ^{\rho }\left( t\right) }{\sigma \left( t\right)
+2t-2s-\rho \left( t\right) } \\
&=\lim_{s\rightarrow t}\frac{f^{\sigma }\left( t\right) -f\left( s\right)
+f\left( 2t-s\right) -f^{\rho }\left( t\right) }{\sigma \left( t\right)
+2t-2s-\rho \left( t\right)}
+ \lim_{s\rightarrow t}\frac{g^{\sigma }\left( t\right)
-g\left( s\right) +g\left( 2t-s\right) -g^{\rho }\left( t\right) }{\sigma
\left( t\right) +2t-2s-\rho \left( t\right) } \\
&= f^{\diamondsuit }\left( t\right) +g^{\diamondsuit }\left( t\right).
\end{split}
\end{equation*}

\noindent (ii) Let $t\in \mathbb{T}_{\kappa}^{\kappa}$
and $\lambda \in \mathbb{R}$. Then,
\begin{equation*}
\begin{split}
\left( \lambda f\right) ^{\diamondsuit }\left( t\right)
&=\lim_{s\rightarrow t}\frac{\left( \lambda f\right) ^{\sigma }\left(
t\right) -\left( \lambda f\right) \left( s\right) +\left( \lambda f\right)
\left( 2t-s\right) -\left( \lambda f\right) ^{\rho }\left( t\right) }{\sigma
\left( t\right) +2t-2s-\rho \left( t\right) } \\
&=\lambda \lim_{s\rightarrow t}\frac{f^{\sigma }\left( t\right) -f\left(
s\right) +f\left( 2t-s\right) -f^{\rho }\left( t\right) }{\sigma \left(
t\right) +2t-2s-\rho \left( t\right)}
=\lambda f^{\diamondsuit }\left( t\right) .
\end{split}
\end{equation*}

\noindent (iii) Let us assume that
$t\in \mathbb{T}_{\kappa}^{\kappa}$
and $f$ and $g$ are continuous at $t$.
If $t$ is dense, then
\begin{equation*}
\begin{split}
\left( fg\right) ^{\diamondsuit }\left( t\right)
&=\lim_{h\rightarrow 0}\frac{\left( fg\right) \left( t+h\right) -\left(
fg\right) \left( t-h\right) }{2h} \\
&=\lim_{h\rightarrow 0}\frac{f\left( t+h\right) -f\left( t-h\right) }{2h}
g\left( t+h\right) +\lim_{h\rightarrow 0}\frac{g\left( t+h\right) -g\left(
t-h\right) }{2h}f\left( t-h\right)  \\
&= f^{\diamondsuit }\left( t\right) g^{\sigma }\left( t\right)
+f^{\rho}\left( t\right) g^{\diamondsuit }\left( t\right) .
\end{split}
\end{equation*}
If $t$ is not dense, then
\begin{equation*}
\begin{split}
\left( fg\right) ^{\diamondsuit }\left( t\right)
&= \frac{\left( fg\right)^{\sigma }\left( t\right)
-\left( fg\right) ^{\rho }\left( t\right) }{\sigma\left( t\right)
-\rho \left( t\right) }
=\frac{f^{\sigma }\left( t\right) -f^{\rho }\left( t\right) }{\sigma
\left( t\right) -\rho \left( t\right) }g^{\sigma }\left( t\right)
+\frac{g^{\sigma }\left( t\right) -g^{\rho }\left( t\right) }{\sigma \left(
t\right) -\rho \left( t\right) }f^{\rho }\left( t\right)  \\
&=f^{\diamondsuit }\left( t\right) g^{\sigma }\left( t\right)
+f^{\rho}\left( t\right) g^{\diamondsuit }\left( t\right)
\end{split}
\end{equation*}
proving the intended equality.

\noindent (iv) Using the relation
$\left( \frac{1}{f}\times f\right) \left( t\right) =1$ we can write that
$0 =\left( \frac{1}{f}\times f\right) ^{\diamondsuit }\left( t\right)
= f^{\diamondsuit }\left( t\right) \left(
\frac{1}{f}\right)^{\sigma}\left( t\right)
+ f^{\rho }\left( t\right) \left(
\frac{1}{f}\right)^{\diamondsuit }\left( t\right)$.
Therefore,
$\left( \frac{1}{f}\right) ^{\diamondsuit }\left( t\right)
=-\frac{f^{\diamondsuit }\left( t\right)}{
f^{\sigma }\left( t\right) f^{\rho}\left( t\right) }$.

\noindent (v) Let $t\in \mathbb{T}_{\kappa}^{\kappa}$. Then,
\begin{equation*}
\begin{split}
\left( \frac{f}{g}\right)^{\diamondsuit }\left( t\right)
&= \left( f\times \frac{1}{g}\right) ^{\diamondsuit }\left( t\right)
= f^{\diamondsuit }\left( t\right) \left( \frac{1}{g}\right)^{\sigma}\left(
t\right) +f^{\rho }\left( t\right) \left( \frac{1}{g}\right)^{\diamondsuit }\left( t\right) \\
&= \frac{f^{\diamondsuit }\left( t\right) }{g^{\sigma }\left( t\right) }
+f^{\rho }\left( t\right) \left(
-\frac{g^{\diamondsuit }\left( t\right)}{g^{\sigma }\left( t\right) g^{\rho }\left( t\right) }\right)
= \frac{f^{\diamondsuit }\left( t\right) g^{\rho }\left( t\right)
-f^{\rho}\left( t\right) g^{\diamondsuit }\left( t\right) }{g^{\sigma }\left(
t\right) g^{\rho }\left( t\right)}.
\end{split}
\end{equation*}
\end{proof}

\begin{example}
The symmetric derivative of $f\left( t\right) =t^{2}$ is
$f^{\diamondsuit }\left( t\right) =\sigma \left( t\right) +\rho \left(t\right)$.
\end{example}

\begin{example}
The symmetric derivative of $f\left( t\right) =1/t$ is
$f^{\diamondsuit }\left( t\right)
=-\frac{1}{\sigma \left( t\right) \rho\left( t\right) }$.
\end{example}


\section{Particular Cases}
\label{ts:sec:fc}

In Section~\ref{ts:sec:dif} we introduced the symmetric derivative on a
time scale $\mathbb{T}$ and derived some of its properties. It has been
shown that the new notion unifies the symmetric derivatives
of classical analysis \cite{Thomson} and quantum calculus \cite{Kac}.
Here we note that our symmetric derivative
is different from the delta and nabla derivatives
considered in the time scale literature \cite{Hilger,Bohner:1,Bohner:2}.
A simple example of a function that is neither delta nor nabla differentiable,
in the sense of time scales, but that is symmetric differentiable,
is the absolute value function.

\begin{example}
\label{ex:dos:nts}
Let $\mathbb{T}$ be a time scale with $0\in \mathbb{T}_{\kappa}^{\kappa}$
and $f:\mathbb{T} \rightarrow \mathbb{R}$
be defined by $f\left( t\right) =\left \vert t\right \vert$.
This function is not differentiable at point $t = 0$
in the sense of time scales \cite{Hilger,Bohner:1,Bohner:2,Sheng}.
However, the symmetric derivative is always well defined:
\begin{equation*}
f^{\diamondsuit }\left( 0\right)
=\lim_{h\rightarrow 0}\frac{f^{\sigma}\left( 0\right)
-f\left( 0+h\right) +f\left( 0-h\right) -f^{\rho }\left(
0\right) }{\sigma \left( 0\right) -2h-\rho \left( 0\right) }
= \lim_{h\rightarrow 0}\frac{\sigma \left( 0\right)
+\rho \left( 0\right) }{\sigma \left( 0\right) -2h-\rho \left( 0\right)},
\end{equation*}
so that $f^{\diamondsuit}(0) = 0$ if $0$ is dense, and
$f^{\diamondsuit}(0) = (\sigma(0)+\rho(0))/(\sigma(0)-\rho(0))$
otherwise.
\end{example}

In the particular case a function is simultaneously
delta and nabla differentiable \cite{Bohner:1,Bohner:2},
Proposition~\ref{ts:delta nabla} below shows that
a relation can be done between our symmetric derivative
and the recent diamond-alpha derivative
\cite{Sheng} (Corollary~\ref{cor:diam}).

\begin{proposition}
\label{ts:delta nabla}
If $f$ is delta and nabla differentiable, then $f$ is symmetric
differentiable and, for each $t\in\mathbb{T}_{\kappa}^{\kappa}$,
$f^{\diamondsuit }\left( t\right)  =\gamma\left( t\right)
f^{\Delta}\left( t\right) +\left(1-\gamma\left( t\right)\right) f^{\nabla }\left( t\right)$,
where
\begin{equation}
\label{eq:gamma}
\gamma\left( t\right)  =\lim_{s\rightarrow t}\frac{\sigma \left(
t\right) -s}{\sigma \left( t\right) +2t-2s-\rho \left( t\right) }.
\end{equation}
\end{proposition}

\begin{proof}
Note that
\begin{equation*}
\begin{split}
f^{\diamondsuit }\left( t\right)
&=\lim_{s\rightarrow t}\frac{f^{\sigma}\left( t\right)
-f\left( s\right) +f\left( 2t-s\right) -f^{\rho }\left(
t\right) }{\sigma \left( t\right) +2t-2s-\rho \left( t\right) } \\
&= \lim_{s\rightarrow t}\bigg{(}\frac{\sigma \left( t\right) -s}{\sigma
\left( t\right) +2t-2s-\rho \left( t\right) }\frac{f^{\sigma }\left(
t\right) -f\left( s\right) }{\sigma \left( t\right) -s}
+\frac{\left( 2t-s\right) -\rho \left( t\right) }{\sigma \left( t\right)
+2t-2s-\rho \left( t\right) }\frac{f\left( 2t-s\right) -f^{\rho }\left(
t\right) }{\left( 2t-s\right) -\rho \left( t\right) }\bigg{)} \\
&= \lim_{s\rightarrow t}\left( \frac{\sigma \left( t\right) -s}{\sigma
\left( t\right) +2t-2s-\rho \left( t\right) }f^{\Delta }\left( t\right)
+\frac{\left( 2t-s\right) -\rho \left( t\right) }{\sigma \left( t\right)
+2t-2s-\rho \left( t\right) }f^{\nabla }\left( t\right) \right) .
\end{split}
\end{equation*}
For each $t\in\mathbb{T}$, define
$\gamma\left( t\right)
:= \lim_{s\rightarrow t}\frac{\sigma \left(
t\right) -s}{\sigma \left( t\right) +2t-2s-\rho \left( t\right) }$
and $\tilde{\gamma}\left( t\right)
:= \lim_{s\rightarrow t}\frac{\left(
2t-s\right) -\rho \left( t\right) }{\sigma \left( t\right) +2t-2s-\rho
\left( t\right)}$.
It is clear that
$\gamma\left( t\right) +\tilde{\gamma}\left( t\right) =1$.
Note that if $t\in \mathbb{T}$ is dense, then
\begin{equation*}
\gamma\left( t\right) = \lim_{s\rightarrow t}\frac{\sigma \left(
t\right) -s}{\sigma \left( t\right) +2t-2s-\rho \left( t\right) }
=\lim_{s\rightarrow t}\frac{t-s}{2t-2s}
=\frac{1}{2}
\end{equation*}
and, therefore, $\tilde{\gamma}\left( t\right) =1/2$.
On the other hand, if $t\in \mathbb{T}$ is not dense, then
\begin{equation*}
\gamma\left( t\right) = \lim_{s\rightarrow t}\frac{\sigma \left(
t\right) -s}{\sigma \left( t\right) +2t-2s-\rho \left( t\right)}
= \frac{\sigma \left( t\right) -t}{\sigma \left( t\right)
-\rho \left(t\right)}
\end{equation*}
and
$\tilde{\gamma}\left( t\right) =\frac{t-\rho \left( t\right) }{\sigma \left(
t\right) -\rho \left( t\right) }$.
Hence, functions $\gamma, \tilde{\gamma} :\mathbb{T}\rightarrow \mathbb{R}$
are well defined and, if $f$ is delta and nabla differentiable, then
$f^{\diamondsuit }\left( t\right) = \gamma\left( t\right)
f^{\Delta}\left( t\right) +\tilde{\gamma}\left( t\right) f^{\nabla }\left( t\right)
= \gamma\left( t\right) f^{\Delta }\left( t\right)
+\left( 1-\gamma\left( t\right) \right) f^{\nabla }\left( t\right)$.
\end{proof}

\begin{remark}
Functions $\gamma, \tilde{\gamma} : \mathbb{T}\rightarrow \mathbb{R}$
are bounded and nonnegative:
$0 \leqslant \gamma(t), \tilde{\gamma}(t) \leqslant 1$.
This is due to the fact that
$\rho \left( t\right)\leqslant t \leqslant\sigma \left( t\right)$
for every $t\in \mathbb{T}$.
\end{remark}

\begin{corollary}
\label{cor:diam}
If $f$ is delta and nabla differentiable
and if function $\gamma(\cdot)$ in \eqref{eq:gamma}
is a constant, $\gamma(t) \equiv \alpha$,
then the symmetric derivative coincides
with the diamond-$\alpha$ derivative:
$f^{\diamondsuit}(t) = \alpha f^{\Delta}\left( t\right)
+(1-\alpha) f^{\nabla}(t)$.
\end{corollary}

In the classical case $\mathbb{T} = \mathbb{R}$ it can be proved that
``A continuous function is necessarily increasing in any interval in which its
symmetric derivative exists and is positive'' \cite{Thomson}.
We note that this result is not valid for the symmetric derivative on time scales.
For instance, consider the time scale $\mathbb{T}=\mathbb{N}$ and function
$f(n) = n$ if $n$ is odd and $f(n) = 10 n$ if $n$ is even.
The symmetric derivative of $f$ is given by
$f^{\diamondsuit }\left( n\right) = \frac{f^{\sigma }\left( n\right)
-f^{\rho }\left( n\right) }{\sigma \left( n\right) -\rho \left( n\right) }
=\frac{10\left( n+1\right) -10\left( n-1\right) }{\left( n+1\right)
-\left( n-1\right)} = 10$
for $n$ odd, while for $n$ even one has
$f^{\diamondsuit }\left( n\right) = \frac{f^{\sigma }\left( n\right)
-f^{\rho }\left( n\right) }{\sigma \left( n\right) -\rho \left( n\right)}
=\frac{\left( n+1\right) -\left( n-1\right) }{\left( n+1\right)
-\left(n-1\right)} = 1$.
Clearly, function $f$ is non-increasing although
its symmetric derivative is always positive.
In this example, the symmetric derivative
coincides with the diamond-$\alpha$ derivative, $\alpha=1/2$,
showing that there is an inconsistency in Corollary~2.1 of \cite{Ozkan}.


\section*{Acknowledgments}

This work was supported by {\it FEDER} funds through {\it COMPETE}
--- Operational Programme Factors of Competitiveness
(``Programa Operacional Factores de Competitividade'')
and by Portuguese funds through the
{\it Center for Research and Development in Mathematics and Applications} (University of Aveiro)
and the Portuguese Foundation for Science and Technology
(``FCT--Funda\c{c}\~{a}o para a Ci\^{e}ncia e a Tecnologia''),
within project PEst-C/MAT/UI4106/2011 with COMPETE number FCOMP-01-0124-FEDER-022690.
Brito da Cruz is also supported by FCT through the Ph.D. fellowship
SFRH/BD/33634/2009. The authors are grateful to the referees for their valuable
comments and helpful suggestions.




\begin{thebibliography}{9}

\bibitem{Hilger}
B. Aulbach\ and\ S. Hilger,
A unified approach to continuous and discrete dynamics,
in {\it Qualitative theory of differential equations (Szeged, 1988)}, 37--56,
Colloq. Math. Soc. J\'anos Bolyai, 53 North-Holland, Amsterdam, 1990.

\bibitem{Bohner:1}
M. Bohner\ and\ A. Peterson,
{\it Dynamic equations on time scales},
Birkh\"auser Boston, Boston, MA, 2001.

\bibitem{Bohner:2}
M. Bohner\ and\ A. Peterson,
{\it Advances in dynamic equations on time scales},
Birkh\"auser Boston, Boston, MA, 2003.

\bibitem{Kac}
V. Kac\ and\ P. Cheung,
{\it Quantum calculus},
Universitext, Springer, New York, 2002.

\bibitem{Ozkan}
U. M. \"Ozkan\ and\ B. Kaymak\c calan,
Basics of diamond-$\alpha$ partial dynamic calculus on time scales,
Math. Comput. Modelling {\bf 50} (2009), no.~9-10, 1253--1261.

\bibitem{Serafin}
R. A. Serafin,
On the symmetric difference quotient and its application
to the correction of orbits,
Celestial Mech. {\bf 26} (1982), no.~4, 383--393.

\bibitem{Sheng}
Q. Sheng, M. Fadag, J. Henderson\ and\ J. M. Davis,
An exploration of combined dynamic derivatives on time scales and their applications,
Nonlinear Anal. Real World Appl. {\bf 7} (2006), no.~3, 395--413.

\bibitem{Thomson}
B. S. Thomson,
{\it Symmetric properties of real functions},
Monographs and Textbooks in Pure and Applied Mathematics,
183, Dekker, New York, 1994.

\bibitem{Zygmund}
A. Zygmund,
{\it Trigonometric series. Vol. I, II},
reprint of the 1979 edition,
Cambridge Mathematical Library,
Cambridge Univ. Press, Cambridge, 1988.

\end{thebibliography}
\end{document}